%% file: main.tex
\documentclass[a4paper,11pt]{scrartcl}
\usepackage{amsmath}
\usepackage{amssymb}
\usepackage{amsthm}

\usepackage{xcolor}
\usepackage{cleveref}

\usepackage{url}

\input{common}

\begin{document}
\title{Lipschitz and H\"older continuity in Reproducing Kernel Hilbert Spaces\footnote{Preprint, currently under review.}}
\author{Christian Fiedler\\
    Institute for Data Science in Mechanical Engineering (DSME)\\
    RWTH Aachen University\\
    Email \url{fiedler@dsme.rwth-aachen.de}
}
\maketitle
\begin{abstract}
Reproducing kernel Hilbert spaces (RKHSs) are very important function spaces, playing an important role in machine learning, statistics, numerical analysis and pure mathematics. 
Since Lipschitz and H\"older continuity are important regularity properties, with many applications in interpolation, approximation and optimization problems, in this work we investigate these continuity notion in RKHSs. We provide several sufficient conditions as well as an in depth investigation of reproducing kernels inducing prescribed Lipschitz or H\"older continuity.
Apart from new results, we also collect related known results from the literature, making the present work also a convenient reference on this topic.
\end{abstract}

\par\vskip\baselineskip\noindent
\par\vskip\baselineskip\noindent
\section{Introduction}
Reproducing kernel Hilbert spaces (RKHSs) are Hilbert function spaces in which evaluation of functions is continuous with respect to (w.r.t.) the Hilbert space norm. 
These function spaces play an important role in machine learning \cite{scholkopf2002learning,shawe2004kernel}, statistics \cite{BTA04}, numerical analysis \cite{wendland2004scattered,fasshauer2015kernel} (including inter alia function interpolation and approximation problems, numerical solution of partial differential equations, and numerical integration), signal processing \cite{garcia2000orthogonal} and pure mathematics \cite{PR16}. 
The theory of RKHSs is by now very well-developed, and there are many excellent expositions available, for example, \cite{aronszajn1950theory,PR16}.
In particular, the connection between properties of the reproducing kernel of an RKHS and properties of the functions in an RKHS has been thoroughly investigated, with a good overview provided in \cite[Chapter~4]{SC08}.
This connection is important since an RKHS is generated by its reproducing kernel (see \Cref{sec:prelims} for the details), and the latter is user-defined in most applications of RKHSs. By choosing or constructing an appropriate reproducing kernel, tailored function spaces can be created, which can then be used in interpolation, approximation, optimization and related problems.

Particularily relevant for many applications, especially in constructive approximation problems, are regularity properties of function spaces.
In the case of RKHSs, continuity and differentiability of functions is fully determined by the corresponding reproducing kernel, cf. \cite[Lemma~4.29,~Corollary~4.36]{SC08}.
Furthermore, there is a close connection between certain Sobolev spaces and RKHSs, cf. \cite{wendland2004scattered,fasshauer2015kernel}.

Another important regularity notion, which is in between mere continuity and differentiability, is Lipschitz continuity, or more generally H\"older continuity.
Recall that if $(\inputset{X},d_\inputset{X})$ and $(\inputset{Y},d_\inputset{Y})$ are two metric spaces, and $f: \inputset{X}\rightarrow\inputset{Y}$ a function, we call $f$ \defm{Lipschitz continuous} if there exists $L\in\Rnn$ such that for all $x,x'\in\inputset{X}$ we have $d_\inputset{Y}(f(x),f(x'))\leq d_\inputset{X}(x,x')$. Each such $L\in\Rnn$ is called a \defm{Lipschitz constant} for $f$, we sometimes we say that $f$ is $L$-Lipschitz continuous.
Similarly, if there exists $\alpha\in\Rp$ and $L_\alpha\in\Rnn$ such that for all $x,x'\in\inputset{X}$ we have $d_\inputset{Y}(f(x),f(x'))\leq L_\alpha d_\inputset{X}(x,x')^\alpha$, then $f$ is called \defm{$\alpha$-H\"older continuous}, and each such $L_\alpha$ is called a \defm{H\"older constant} for $f$.
In particular, 1-H\"older continuity is Lipschitz continuity.

Lipschitz and H\"older continuity are classic notions that appear prominently for example in the theory of ordinary differential equations \cite{amann2011ordinary} and partial differential equations \cite{evans2022partial}, respectively. H\"older continuity is also frequently used in the theory of nonparametric statistics \cite{tsybakov2009nonparametric,ghosal2017fundamentals}.
Moreover, there is now a considerable and well-developed theory of spaces of Lipschitz continuous functions, cf. \cite{cobzacs2019lipschitz}.
Finally, Lipschitz continuity (and to a lesser extent also H\"older continuity) is used as the foundation of practical algorithms.
For example, Lipschitz continuity (and a known Lipschitz constant) is a core assumption in many global optimization approaches \cite{pinter1995global}.
Lipschitz continuity also forms the basis for many non-stochastic learning algorithms, especially in the context of systems identification \cite{milanese2004set,calliess2014conservative}.
Recently, Lipschitz assumptions have also been used successfully in the context of kernel methods, for example, for Bayesian optimization with safety constraints \cite{sui2015safe}, or function approximation and regression problems with bounded noise \cite{fiedler2022learning}, the latter motivated by the stringent requirements of learning-based robust control, cf. \cite{fiedler2021practical} and \cite{fiedler2021learning} for an in depth discussion of this issue.

All of this forms a strong motivation to investigate Lipschitz and H\"older continuity in RKHSs. In particular, a central question is how (if at all) the Lipschitz or H\"older continuity of the reproducing kernel of an RKHS influences the corresponding continuity properties of RKHS functions.
To the best of our knowledge, there is no systematic investigation into these questions, despite the importance of RKHSs and Lipschitz  and H\"older continuity, respectively, and the considerable effort that went into investigating the connection between kernel properties and RKHS function properties.

That RKHS functions are always Lipschitz continuous w.r.t. the kernel metric, as reviewed in \Cref{sec:lipschitz:kernelMetric}, is well-known.
The more interesting question of Lipschitz and H\"older continuity w.r.t. an arbitrary metric seems to have been barely covered in the literature.
The only previous work we are aware of that explicitly addressing this question, is \cite{ferreira2013positive}.
In the present work, we are closing this gap in the literature. 

\paragraph{Outline and contributions} 
We provide a comprehensive account on Lipschitz and H\"older continuity in RKHSs. On the one hand, this includes a collection of (the relatively few) known results, and on the other hand a systematic investigation of this issue, including characterization and converse results.

In \Cref{sec:prelims}, we recall fundamental results on RKHSs and introduce our notation.

\Cref{sec:lipschitz:kernelMetric} is concerned with Lipschitz continuity w.r.t. the kernel metric induced by the unique reproducing kernel of an RKHS. Most of the results there are known, however, since we are not aware of a systematic exposition thereof, we provide all the details for ease of future reference.

In \Cref{sec:lipschitz:hoelderContMetricSpace}, we investigate H\"older and Lipschitz continuity w.r.t. a given metric. 
First, some preliminary facts regarding bivariate H\"older and Lipschitz continuous functions are provided, some of these seem to have been not noticed before.

We then investigate which continuity properties in an RKHS are induced by H\"older continuous kernels. While the principle arguments are contained already in \cite{ferreira2013positive}, our results are more general and easier to state.
Finally, the converse question is tackled: If all RKHS functions fulfill a H\"older continuity property, what does this mean for the reproducing kernel?
To the best of our knowledge, this problem has not been dealt with before.

One key take-away of \Cref{sec:lipschitz:hoelderContMetricSpace} is the fact that a Lipschitz continuous kernel does not directly lead to Lipschitz continuous RKHS functions. Since Lipschitz continuous functions are desirable in many applications, it would be interesting to construct kernels that induce such RKHS functions.
\Cref{sec:lipschitz:inducing} is concerned with this problem. 
First, we give a characterization of kernels that induce H\"older continuous RKHS functions, a result which is completely new.
Next, we give sufficient conditions in terms of certain integral operators, extending a result from \cite{ferreira2013positive}.
Finally, we give a very general construction based on feature mixtures, vastly generalizing a method from \cite{wu2018d2ke}.

We close in \Cref{sec:conclusion} with a summary and discussion of our results, as well as an outlook to applications and future research directions.

\section{Preliminaries and Background} \label{sec:prelims}
We cover the real and complex case simultaneously, using the symbol $\K$ for $\R$ or $\C$.
Unless noted otherwise, $\inputset{X}$ will be a non-empty set.
We call $\kappa:\inputset{X}\times\inputset{X}\rightarrow\K$ \defm{Hermitian} if for all $x,x'\in\inputset{X}$, we have $\kappa(x,x')=\overline{\kappa(x',x)}$. Note that if $\kappa$ is Hermitian, then $\kappa(x,x)\in\R$ for all $x\in\inputset{X}$.
If $\K=\R$, then $\kappa$ is Hermitian if and only if it is symmetric in its two arguments.

Let us recall some important definitions and facts about RKHSs, following mostly \cite[Chapter~4]{SC08}. Consider a function $k:\inputset{X}\times\inputset{X}\rightarrow\K$, and let $H\subseteq \K^\inputset{X}$ be a Hilbert space of functions on $\inputset{X}$.

We call $k$ a \defm{kernel} (or \defm{$\K$-kernel}) \defm{on $\inputset{X}$} if there exists a $\K$-Hilbert space $\fs$ and a map $\fm:\inputset{X}\rightarrow\fs$ such that
\begin{equation}
    k(x,x') = \langle \Phi(x'), \Phi(x)\rangle_\fs \quad \forall x,x'\in\inputset{X}.
\end{equation}
In this case, we call $\fs$ a \defm{feature space} and $\fm$ a \defm{feature map} for $k$.

The function $k$ is called \defm{positive semidefinite}\footnote{The terminology is not uniform in the literature. Other common terms are \defm{of positive type} and \defm{positive definite}.} if for all $N\in\Np$ and $x_1,\ldots,x_N\in\inputset{X}$, the matrix $(k(x_j,x_i))_{i,j=1,\ldots,N}$ is positive semidefinite in the sense of linear algebra.

We call $H$ a \defm{reproducing kernel Hilbert space (RKHS)}, if for all $x\in\inputset{X}$ the evaluation functionals $\delta_x: H\rightarrow \K$, $f \mapsto f(x)$, are continuous w.r.t. the topology induced by the scalar product of $H$.

The function $k$ is called a \defm{reproducing kernel} for or of $H$, if for all $x\in\inputset{X}$, $k(\cdot,x)\in H$, and for all $f\in H$, $x\in\inputset{X}$, it holds that $f(x)=\langle f, k(\cdot,x)\rangle_H$.

Let us recall some basic facts about RKHSs. The function $k$ is a kernel if and only if it is positive semidefinite. 
The Hilbert space of functions $H$ has a reproducing kernel if and only if $H$ is an RKHS.
In this case, the reproducing kernel is unique and a kernel (and hence also positive semidefinite). Furthermore, $H$ is a feature space for $k$, and $\Phi_k: \inputset{X}\rightarrow H$, $\Phi_k(x)=k(\cdot,x)$ is a feature map for $k$, called the \defm{canonical feature map} of $k$.
Finally, a positive semidefinite $k$ is the reproducing kernel of a uniquely determined Hilbert space of functions, which we denote by $(H_k,\scp_k)$, and the latter is an RKHS.
In particular, the terms kernel, reproducing kernel, and positive semidefinite are equivalent in the context of RKHSs.

Given a positive semidefinite $k$ and its associated RKHS $H_k$, define the \defm{pre-RKHS}
\begin{equation}
    \Hpre{k} = \{ k(\cdot,x) \mid x \in\inputset{X} \}
    = \left\{ \sum_{n=1}^N \alpha_n k(\cdot, x_n) \mid \alpha_1,\ldots,\alpha_N\in\K,\: x_1,\ldots,x_N\in\inputset{X} \right\}.
\end{equation}
It is well-known that for $f,g\in\Hpre{k}$ with representations $f= \sum_{n=1}^N \alpha_n k(\cdot, x_n)$, $g= \sum_{m=1}^M \beta_m k(\cdot, y_m)$,
\begin{equation}
    \langle f, g\rangle_k = \sum_{n=1}^N \sum_{m=1}^M \alpha_n \overline{\beta_m} k(y_m, x_n),
\end{equation}
and $\Hpre{k}$ is dense in $H_k$.

Let $k$ be a kernel on $\inputset{X}$, and $(\fs,\fm)$ a corresponding feature space-feature map pair, then
\begin{equation}
    d_\fm: \inputset{X}\times\inputset{X}\rightarrow\Rnn, \quad d_\fm(x,x')=\|\fm(x)-\fm(x')\|_\fs
\end{equation}
is a semimetric on $\inputset{X}$. If $(\fs,\fm)=(H_k,\Phi_k)$, we set $d_k=d_{\Phi_k}$ and call this the \defm{kernel (semi)metric}.

The next result is well-known, but rarely explicitly stated.
\begin{lemma}
Let $k:\inputset{X}\times\inputset{X}\rightarrow\K$ be a kernel on $\inputset{X}\not=\emptyset$. Then for all  feature space-feature map pairs $(\fs,\fm)$, we have $d_\fm=d_k$.
\end{lemma}
When working with $d_k$, this result allows us to work with $d_\fm$ instead, where $\fm$ is any feature map, and vice versa.
\begin{proof}
Let $(\fs,\fm)$ be a feature space-feature map pair, and $x,x'\in\inputset{X}$ be arbitrary. We then have
\begin{align*}
    d_\fm(x,x') & = \|\fm(x)-\fm(x')\|_\fs \\
    & = \sqrt{\langle \fm(x)-\fm(x'), \fm(x)-\fm(x')\rangle_\fs } \\
    & = \sqrt{\langle \fm(x), \fm(x)\rangle_\fs + \langle \fm(x), \fm(x')\rangle_\fs + \langle \fm(x'), \fm(x)\rangle_\fs  + \langle \fm(x'), \fm(x')\rangle_\fs} \\
    & = \sqrt{k(x,x) + k(x,x') + k(x',x) + k(x',x')} \\
    & = \sqrt{\langle k(\cdot,x) - k(\cdot,x'),  k(\cdot,x) - k(\cdot,x') \rangle_k} \\
    & = \|\Phi_k(x) - \Phi_k(x')\|_k \\
    & = d_k(x,x'),
\end{align*}
establishing the claim.
\end{proof}
Since we state several results for bounded kernels or bounded RKHS functions, we recall the following characterization of boundedness in RKHSs.
\begin{lemma} \label{properties:lipschitz:boundednessEquiv}
Let $\inputset{X}\not=\emptyset$ be some set and $k:\inputset{X}\times\inputset{X}\rightarrow\K$ a kernel on $\inputset{X}$. The following statements are equivalent.
\begin{enumerate}
    \item $k$ is bounded
    \item $\|k\|_\infty := \sup_{x\in \inputset{X}} \sqrt{k(x,x)} < \infty$
    \item There exists a feature space-feature map pair $(\fs,\fm)$ such that $\fm$ is bounded
    \item For all feature space-feature map pairs $(\fs,\fm)$, $\fm$ is bounded
    \item All $f\in H_k$ are bounded 
\end{enumerate}
If any of the statements is true, then for all feature space-feature map pairs $(\fs,\fm)$, we have $\|k\|_\infty=\sup_{x\in\inputset{X}}\|\fm(x)\|_\fs$, and $|f(x)|\leq \|f\|_k \|k\|_\infty$, for all $f\in H_k$ and $x\in\inputset{X}$.
\end{lemma}
\begin{proof}
Let $(\fs,\fm)$ be any feature space-feature map. For $x,x'\in\inputset{X}$ we have
\begin{equation*}
    |k(x,x')|=|\langle \fm(x'), \fm(x) \rangle_\fs| \leq \|\fm(x')\|_\fs \|\fm(x)\|_\fs = \sqrt{k(x',x')}\sqrt{k(x,x)},
\end{equation*}
and the equivalence of the first four items is now clear.
The equivalence between the first and last item is provided by \cite[Lemma~4.23]{SC08}.

Finally, since for any feature space-feature map pair $(\fs,\fm)$, and all $x\in\inputset{X}$, we have $\sqrt{k(x,x)}=\|\fm(x)\|_\fs$, and for all $f\in H_k$ we have $|f(x)|=|\langle f, k(\cdot,x)\rangle_k|\leq \|f\|_k \sqrt{k(x,x)}$, the last assertion follows.
\end{proof}
Finally, we recall the following result on Parseval frames in an RKHS, which corresponds to \cite[Theorem~2.10,~Exercise~3.7]{PR16}, and is called \emph{Papadakis Theorem} there.
\begin{theorem} \label{thm:papadakis}
Let $\inputset{X}\not=\emptyset$ be a set and $k:\inputset{X}\times\inputset{X}\rightarrow\K$ a kernel on $\inputset{X}$.
\begin{enumerate}
    \item If $(f_i)_{i\in I}$ is a Parseval frame in $H_k$, then for all $x,x'\in\inputset{X}$
    \begin{equation}
        k(x,x')=\sum_{i \in I} f_i(x) \overline{f_i(x')},
    \end{equation}
    where the convergence is pointwise.
    \item Consider a family of functions $(f_i)_{i\in I}$, where $f_i \in \K^\inputset{X}$ for all $i\in I$, such that
    \begin{equation}
        k(x,x')=\sum_{i \in I} f_i(x) \overline{f_i(x')}
    \end{equation}
    for all $x,x'\in\inputset{X}$, where the convergence is pointwise. Then $f_i\in H_k$ for all $i\in I$, and $(f_i)_{i\in I}$ is a Parseval frame in $H_k$.
\end{enumerate}
\end{theorem}

\input{lipschitz.tex}

\section{Conclusion} \label{sec:conclusion}
We presented a comprehensive discussion of Lipschitz and H\"older continuity of RKHS functions. Starting with the well-known Lipschitz continuity w.r.t. the kernel (semi)metric, we then investigated H\"older-continuity w.r.t. a given metric, including converse results, i.e., consequences of H\"older continuity in function spaces related to RKHSs.
Finally, we provided characterizations as well as sufficient conditions for kernels inducing prescribed Lipschitz and H\"older continuity of their RKHS functions w.r.t. a given metric, an important aspect for applications.

The results presented here can be used to construct tailored kernels ensuring Lipschitz or H\"older continuous RKHS functions, or to check that existing kernels have such RKHS functions. Furthermore, because the results are \emph{quantitative}, they can be used in numerical methods. In particular, we are currently investigating their application in methods like \cite{fiedler2022learning} and \cite{sui2015safe}.

Finally, we would like to point out three interesting questions for future work. 

First, the Lipschitz and H\"older continuity in RKHS that we have been concerned with here, are of a strong \emph{uniform} nature, since the corresponding Lipschitz or H\"older constants are proportional to the RKHS function of the respective function, cf. the developments in \Cref{sec:lipschitz:hoelderContMetricSpace}. It would be interesting to investigate whether there exist kernels that enforce weaker, nonuniform Lipschitz or continuity properties.

Second, we investigated sufficient conditions for Lipschitz and H\"older continuity of RKHS functions via integral operators. However, all statements are restricted to the range space of the involved integral operators. Under some conditions, these range spaces are dense in RKHSs, so it would be interesting to investigate whether the Lipschitz and H\"older continuity properties transfers to the whole RKHS. Note that this is not trivial since in the H\"older constant in \Cref{properties:lipschitz:hoelderContinuityIntOpRange} involves the $L^q$-norm of the preimage function, not the RKHS norm of the image function.

Finally, the results in \Cref{sec:lipschitz:hoelderContMetricSpaceSuffCond} provide Lipschitz or H\"older constants involving the RKHS norm. However, it is unclear how conservative these results are, i.e., how much larger the Lipschitz or H\"older constants are compared to the best possible constants. Intuitively, it is clear that for generic RKHS functions there will be some conservatism. It would be interesting to investigate how big this conservatism is, and how it depends on properties of the kernel.

\bibliographystyle{plain} 
\bibliography{refs} 

\end{document}

%% file: common.tex
\newcommand{\defm}[1]{\emph{#1}}

\theoremstyle{definition}
\newtheorem{definition}{Definition}[section]
\newtheorem{proposition}[definition]{Proposition}
\newtheorem{theorem}[definition]{Theorem}
\newtheorem{corollary}[definition]{Corollary}
\newtheorem{remark}[definition]{Remark}

\newtheorem{lemma}[definition]{Lemma}
\newtheorem{assumption}[definition]{Assumption}

\newtheorem*{remark*}{Remark}
\newtheorem*{lemma*}{Lemma}

\newtheorem*{internalnote*}{Internal note}

\newcommand{\R}{\mathbb{R}}
\newcommand{\C}{\mathbb{C}}

\newcommand{\Rnn}{\mathbb{R}_{\geq0}}

\newcommand{\Rp}{\R_{>0}}

\newcommand{\Np}{\mathbb{N}_+}

\newcommand{\vspan}{\mathrm{span}}

\newcommand{\ball}{\mathcal{B}}

\newcommand{\identity}{\mathrm{id}} %

\newcommand{\Ellp}{\mathcal{L}}

\newcommand{\E}{\mathbb{E}}

\newcommand{\inputset}[1]{\mathcal{#1}}
\newcommand{\K}{\mathbb{K}}
\newcommand{\fs}{\mathcal{H}}
\newcommand{\fm}{\Phi}
\newcommand{\setsys}[1]{\mathcal{#1}}

\newcommand{\scp}{\langle \cdot, \cdot\rangle}

\newcommand{\Hpre}[1]{H^\text{pre}_{#1}}

%% file: lipschitz.tex
\section{Lipschitz continuity and the kernel metric} \label{sec:lipschitz:kernelMetric}
We just saw that a kernel $k$ on an arbitrary set $\inputset{X}\not=\emptyset$ metrizes this set through the kernel (semi)metric $d_k$.
Note that this holds for \emph{any} set $\inputset{X}$, no matter whether it has additional structure on it or not.
It is therefore natural to investigate Lipschitz continuity of RKHS functions w.r.t. the kernel metric.
We start with the following classic result, which seems to be folklore. %
\begin{proposition} \label{properties:lipschitz:rkhsFuncslipschitzContinuousKernelMetric}
    Let $\inputset{X}\not=\emptyset$ be some set, $k:\inputset{X}\times\inputset{X}\rightarrow\K$ a kernel on $\inputset{X}$, and $d_k$ the corresponding kernel (semi)metric.
    For all $f\in H_k$, we have that $f$ is Lipschitz continuous w.r.t. $d_k$ with Lipschitz constant $\|f\|_k$.
\end{proposition}
In other words, RKHS functions are always Lipschitz continuous w.r.t. the kernel (semi)metric, and their RKHS norm is a Lipschitz constant. This reinforces the intuition that the RKHS norm is a measure of complexity or smoothness of an RKHS function w.r.t. a kernel: The smaller the RKHS norm, the smaller the Lipschitz bound of an RKHS function w.r.t. to the kernel (semi)metric.
\begin{proof}
    Let $f\in H_k$ and $x,x'\in \inputset{X}$ be arbitrary, then
    \begin{equation*}
        |f(x)-f(x')| = |\langle f, k(\cdot,x) - k(\cdot,x')\rangle_k| \leq \|f\|_k \|k(\cdot,x)-k(\cdot,x')\|_k = \|f\|_k d_k(x,x')
    \end{equation*}
\end{proof}
The next result seems to be less well-known. Parts of it can be found for example in \cite[Proposition~2.4]{alpay2021new}.
\begin{proposition}
Let $\inputset{X}\not=\emptyset$ be some set, and $k:\inputset{X}\times\inputset{X}\rightarrow\K$ a kernel on $\inputset{X}$.
\begin{enumerate}
    \item The function $k(\cdot,x)\in H_k$ is Lipschitz continuous w.r.t. $d_k$ with Lipschitz constant $\sqrt{k(x,x)}$, for all $x\in\inputset{X}$.
    \item For all $x_1,x_1',x_2,x_2'\in\inputset{X}$, {\tiny
    \begin{align}
        |k(x_1,x_2)-k(x_1',x_2')| \leq \min\left\{
            \max\left\{\sqrt{k(x_2,x_2)}, \sqrt{k(x_1',x_1')} \right\}, 
            \max\left\{\sqrt{k(x_1,x_1)}, \sqrt{k(x_2',x_2')} \right\}
        \right\}(d_k(x_1,x_1')+ d_k(x_2,x_2')).
    \end{align} }
    If $k$ is bounded, then it is Lipschitz continuous w.r.t. the product metric on $\inputset{X}\times\inputset{X}$ with Lipschitz constant $\|k\|_\infty$.
    \item For all $x,x'\in\inputset{X}$,
    \begin{equation}
        |k(x,x)-k(x',x')| \leq 2\max\{\sqrt{k(x,x)}, \sqrt{k(x',x')}\}d_k(x,x').
    \end{equation}
    If $k$ is bounded, then $x \mapsto k(x,x)$ is Lipschitz continuous w.r.t. $d_k$ with Lipschitz constant $2\|k\|_\infty$.
    \item The function $x \mapsto \sqrt{k(x,x)}$ is Lipschitz continuous w.r.t. $d_k$ and 1 is a Lipschitz constant.
    \item If $(\fs,\fm)$ is any feature space-feature map-pair, then $\Phi$ is Lipschitz continuous w.r.t. $d_k$ with Lipschitz constant 1.
\end{enumerate}
\end{proposition}
\begin{proof}
The first item follows immediately from \Cref{properties:lipschitz:rkhsFuncslipschitzContinuousKernelMetric} clear since $\|k(\cdot,x)\|_k=\sqrt{k(x,x)}$.

To show the second item, let $x_1,x_1',x_2,x_2'\in\inputset{X}$, then
\begin{align*}
    |k(x_1,x_2)-k(x_1',x_2')| & \leq | |k(x_1,x_2)-k(x_1',x_2)| + |k(x_1',x_2) - k(x_1',x_2')| \\
    & = |k(x_1,x_2)-k(x_1',x_2)| + |k(x_2,x_1') - k(x_2',x_1')| \\
    & \leq \sqrt{k(x_2,x_2)}d_k(x_1,x_1') + \sqrt{k(x_1',x_1')}d_k(x_2,x_2') \\
    & \leq \max\left\{\sqrt{k(x_2,x_2)}, \sqrt{k(x_1',x_1')} \right\}(d_k(x_1,x_1')+ d_k(x_2,x_2')).
\end{align*}
Repeating this computation with $x_1, x_2'$ instead of $x_2,x_1'$ establishes the claim.

The next item is now an immediate consequence.

For the second to last item, let $x,x'\in\inputset{X}$, then the converse triangle inequality (in $H_k$) leads to
\begin{equation*}
    |\sqrt{k(x,x)}-\sqrt{k(x',x')}| = |\|k(\cdot,x)\|_k - \|k(\cdot,x')\|_k| \leq \| k(\cdot,x)-k(\cdot,x')\| = d_k(x,x'),
\end{equation*}
so $x \mapsto \sqrt{k(x,x)}$ is indeed 1-Lipschitz w.r.t. $d_k$.

The last item is clear.
\end{proof}
\section{Lipschitz and H\"older continuity on metric spaces} \label{sec:lipschitz:hoelderContMetricSpace}
As we recalled in the preceding section, RKHS functions are always Lipschitz continuous w.r.t. the kernel (semi)metric. 
However, this metric is in general independent of any additional structure on the input set. In particular, if the input set is already a metric space, then this structure is essentially ignored by the kernel (semi)metric. 
In many applications, we are given a metric space as input set, and we would like to have Lipschitz or H\"older continuity of RKHS functions w.r.t. to the existing metric on the input space. We will now investigate this question in depth.
\subsection{Preliminaries}
Since kernels are special bivariate functions, we present some preliminary material on H\"older and Lipschitz continuity of general functions of two variables. Everything in this subsection is elementary and probably known, but we could not locate explicit references, hence we provide all the details.

Let $(\inputset{X},d_\inputset{X})$ be a metric space and $\kappa:\inputset{X}\times\inputset{X}\rightarrow\K$ some function.
\begin{lemma} \label{properties:lipschitz:hoelderImpliesSeperateHoelder}
Assume that there exist a constant $\alpha\in\Rp$, some function $L_\alpha: \inputset{X}\rightarrow\Rnn$, and for all $x\in\inputset{X}$ a set $U_x\subseteq\inputset{X}$ with $x\in U_x$, such that for all $x_1,x_1',x_2,x_2'\in\inputset{X}$ we have
\begin{equation}
|\kappa(x_1,x_2)-\kappa(x_1',x_2')| \leq L_\alpha(x)(d_\inputset{X}(x_1,x_1')^\alpha + d_\inputset{X}(x_2,x_2')^\alpha).
\end{equation}
\begin{enumerate}
    \item For all $x_2\in\inputset{X}$ and all $x_1,x_1'\in U_{x_2}$, we have that
    \begin{equation}
    |\kappa(x_1,x_2)-\kappa(x_1',x_2)| \leq L_\alpha(x)d_\inputset{X}(x_1,x_1')^\alpha.
    \end{equation}
    \item Assume furthermore that $\kappa$ is Hermitian. We then have for all $x\in\inputset{X}$ and $x'\in U_x$ with $x\in U_{x'}$ that
    \begin{equation}
        |\kappa(x)-\kappa(x')|\leq (L_\alpha(x)+L_\alpha(x')) d_\inputset{X}(x,x')^\alpha,
    \end{equation}
    where we defined $\kappa(x):=\kappa(x,x)$.
\end{enumerate}
\end{lemma}
\begin{proof}
The first claim is trivial.
For the second, let $x\in\inputset{X}$ and $x'\in U_x$ be arbitrary, then we have
\begin{align*}
    |\kappa(x)-\kappa(x')| & = |\kappa(x,x)-\kappa(x',x')| \\
    & \leq |\kappa(x,x)-\kappa(x',x)| + |\kappa(x',x)-\kappa(x',x')| \\
    & = |\kappa(x,x)-\kappa(x',x)| + |\kappa(x,x')-\kappa(x',x')| \\
    & \leq (L_\alpha(x) + L_\alpha(x'))d_\inputset{X}(x,x')^\alpha,
\end{align*}
where we used $|\kappa(x',x)-\kappa(x',x')|=|\overline{\kappa(x,x')}-\overline{\kappa(x',x')}|=|\kappa(x,x')-\kappa(x',x')|$ in the second equality.
\end{proof}
\begin{lemma} \label{properties:lipschitz:separateHoelderAndHermitianImpliesHoelder}
Assume that there exist a constant $\alpha\in\Rp$, some function $L_\alpha: \inputset{X}\rightarrow\Rnn$, and for all $x\in\inputset{X}$ a set $U_x\subseteq\inputset{X}$ with $x\in U_x$, such that for all $x_1,x_1'\in\inputset{X}$ we have
\begin{equation}
    |\kappa(x_1,x)-\kappa(x_1',x)| \leq L_\alpha(x)d_\inputset{X}(x_1,x_1')^\alpha.
\end{equation}
If $\kappa$ is Hermitian, then we have for all $x_1,x_1',x_2,x_2'\in\inputset{X}$ with $x_1,x_1'\in U_{x_2}$ and $x_2,x_2'\in U_{x_1'}$ that
\begin{equation}
    |\kappa(x_1,x_2) - \kappa(x_1',x_2')| \leq L_\alpha(x_2) d_\inputset{X}(x_1,x_1')^\alpha + L_\alpha(x_1') d_\inputset{X}(x_2,x_2')^\alpha.
\end{equation}
\end{lemma}
\begin{proof}
Let $x_1,x_1',x_2,x_2'\in\inputset{X}$ such that $x_1,x_1'\in U_{x_2}$ and $x_2,x_2'\in U_{x_1'}$, then we get
\begin{align*}
    |\kappa(x_1,x_2) - \kappa(x_1',x_2')| & \leq |\kappa(x_1,x_2) - \kappa(x_1',x_2)| + |\kappa(x_1',x_2) -  \kappa(x_1',x_2')| \\
    & =  |\kappa(x_1,x_2) - \kappa(x_1',x_2)| + |\overline{\kappa(x_2,x_1')} - \overline{\kappa(x_2',x_1')}| \\
    & = |\kappa(x_1,x_2) - \kappa(x_1',x_2)| + |\kappa(x_2,x_1') - \kappa(x_2',x_1')| \\
    & \leq L_\alpha(x_2) d_\inputset{X}(x_1,x_1')^\alpha + L_\alpha(x_1') d_\inputset{X}(x_2,x_2')^\alpha.
\end{align*}
\end{proof}
We now consider the special case of Lipschitz continuity, corresponding to $\alpha=1$ in the preceding results.
\begin{definition}
We call $\kappa$ \defm{Lipschitz continuous in the first argument with Lipschitz constant $L\in\Rnn$}, or \defm{$L$-Lipschitz continuous in the first argument}, if for all $x_1,x_1',x_2\in\inputset{X}$ we have
\begin{equation}
    |\kappa(x_1,x_2)-\kappa(x_1',x_2)| \leq L d_\inputset{X}(x_1,x_1').
\end{equation}
Similarly, we define $L$-Lipschitz-continuity in the second argument.
Finally, we call $\kappa$ \defm{separately $L$-Lipschitz continuous} if it is $L$-Lipschitz continuous in the first and the second coordinate.
\end{definition}
\begin{proposition} \label{properties:lipschitz:equivLipContForHermitian}
Let $\kappa$ be Hermitian, then the following statements are equivalent.
\begin{enumerate}
    \item $\kappa$ is $L$-Lipschitz continuous (w.r.t. the product metric on $\inputset{X}\times\inputset{X}$)
    \item $\kappa$ is $L$-Lipschitz continuous in the first argument
    \item $\kappa$ is $L$-Lipschitz continuous in the second argument
    \item $\kappa$ is separately $L$-Lipschitz continuous
\end{enumerate}
\end{proposition}
\begin{proof}
By definition, if $\kappa$ is separately $L$-Lipschitz continuous, it is $L$-Lipschitz continuous in the first and second argument.
Since $\kappa$ is Hermitian, the equivalence of items 2 and 3 are clear, so any one these two items implies the fourth item.
\Cref{properties:lipschitz:hoelderImpliesSeperateHoelder} shows that item 1 implies item 4.
Finally, \Cref{properties:lipschitz:separateHoelderAndHermitianImpliesHoelder} shows that item 2 implies item 1.
\end{proof}
Since kernels are always Hermitian, \Cref{properties:lipschitz:equivLipContForHermitian} immediately leads to the following result.
\begin{corollary} \label{properties:lipschitz:LipContSepLipContEquivForKernels}
Let $k: \inputset{X}\times\inputset{X}\rightarrow\K$ be a kernel, and $L\in\Rnn$. $k$ is $L$-Lipschitz continuous if and only if it is separately $L$-Lipschitz continuous.
\end{corollary}
Why is \Cref{properties:lipschitz:LipContSepLipContEquivForKernels} interesting? Let $\inputset{X}$ be a topological space and $k$ a kernel on $\inputset{X}$. It is well-known that $k$ is continuous if and only if it is separately continuous, i.e., $k(\cdot,x)$ is continuous for all $x\in\inputset{X}$, and $x\mapsto k(x,x)$ is continuous, cf. \cite[Lemma~4.29]{SC08}.
In particular, separate continuity of $k$ is not enough for $k$ to be continuous.
For example, there exists a kernel on $\inputset{X}=[-1,1]$ that is bounded and separately continuous, but not continuous, cf. \cite{lehto1952some}.
\Cref{properties:lipschitz:LipContSepLipContEquivForKernels} asserts that in contrast to continuity, \emph{Lipschitz continuity} is equivalent to separate Lipschitz continuity for kernels.

\subsection{RKHS functions of H\"older-continuous kernels} \label{sec:lipschitz:hoelderContMetricSpaceSuffCond}
We now investigate how H\"older continuity of the kernel induces H\"older continuity of RKHS functions.
We start with the following very general result, which covers essentially all potentially relevant forms of Lipschitz and H\"older continuity. It is a generalization of \cite[Proposition~5.2]{ferreira2013positive}.
\begin{theorem} \label{properties:lipschitz:hoelderContinuousKernelsAndRKHSfuncs}
Let $(\inputset{X},d_\inputset{X})$ be a metric space and $k:\inputset{X}\times\inputset{X}\rightarrow\K$ a kernel. Let $\alpha\in\Rp$ and assume that there exist a function $L_\alpha:\inputset{X}\rightarrow\Rnn$ and for each $x\in\inputset{X}$ a set $U_x\subseteq\inputset{X}$ with $x\in U_x$, such that for all $x_1,x_1'\in U_x$ we have
\begin{equation}
    |k(x_1,x)-k(x_1',x)| \leq L_\alpha(x)d_\inputset{X}(x_1,x_1')^\alpha.
\end{equation}
\begin{enumerate}
    \item Let $(\fs,\fm)$ be an arbitrary feature space-feature map-pair for $k$. For all $x,x'\in\inputset{X}$ with $x' \in U_x$ we have
    \begin{equation}
        \|\Phi(x)-\Phi(x')\|_\fs \leq \sqrt{2L_\alpha(x)}d_\inputset{X}(x,x')^{\frac{\alpha}{2}}.
    \end{equation}
    \item For all $f\in H_k$ and $x,x'\in\inputset{X}$ with $x' \in U_x$ we have
    \begin{equation}
        |f(x)-f(x')| \leq \sqrt{2 L_\alpha(x)}\|f\|_k d_\inputset{X}(x,x')^{\frac{\alpha}{2}}.
    \end{equation}
\end{enumerate}
\end{theorem}
\begin{proof}
Let $x,x'\in\inputset{X}$ with $x'\in U_x$ be arbitrary.
If $(\fs,\fm)$ is a feature space-feature map-pair for $k$, then we get
\begin{align*}
    \|\Phi(x)-\Phi(x')\|_\fs & = d_\Phi(x,x') = d_k(x,x') \\
    & = \sqrt{k(x,x)+k(x',x')-k(x,x')-k(x',x)} \\
    & \leq  \sqrt{|k(x,x)-k(x',x)| + |k(x,x')-k(x',x')|} \\
    & \leq \sqrt{2L_\alpha(x) d_\inputset{X}(x,x')^\alpha},
\end{align*}
where we used in the last inequality that $x'\in U_x$.

Let now $f\in H_k$, then we have
\begin{align*}
    |f(x)-f(x')| & \leq \|f\|_k \|k(\cdot,x)-k(\cdot,x')\|_k \\
    & \leq \sqrt{2L_\alpha(x)} \|f\|_kd_\inputset{X}(x,x')^{\frac{\alpha}{2}},
\end{align*}
where we used that $(H_k,\Phi_k)$ is a feature space-feature map-pair for $k$.
\end{proof}
For convenience, we record the following special case.
\begin{corollary}
Let $(\inputset{X},d_\inputset{X})$ be a metric space and $k:\inputset{X}\times\inputset{X}\rightarrow\K$ a kernel that is separately $L$-Lipschitz continuous, then for every $f\in H_k$ and $x,x'\in\inputset{X}$ we have
\begin{equation}
    |f(x)-f(x')| \leq \sqrt{2L} \sqrt{d_\inputset{X}(x,x')}.
\end{equation}
\end{corollary}
\begin{remark}
Consider the situation of \Cref{properties:lipschitz:hoelderContinuousKernelsAndRKHSfuncs}.
\begin{enumerate}
\item If $\alpha\in(0,1)$, $\delta\in\Rp$, $U_x=\ball_\delta(x)$ and $L_\alpha \equiv L_k$ for some $L_k\in\Rnn$, then we recover \cite[Proposition~5.2]{ferreira2013positive}. 
\item If $\alpha\in(0,1)$, $U_x=\inputset{X}$ for all $x\in\inputset{X}$, $L_\alpha \equiv L_k$ for some $L_k\in\Rnn$, then we get that for $f\in H_k$ and $x,x'\in \inputset{X}$ that
\begin{equation*}
    |f(x)-f(x')| \leq \sqrt{2 L_k}\|f\|_k d_\inputset{X}(x,x')^{\frac{\alpha}{2}}
\end{equation*}
We can describe this as "A separately $\alpha$-H\"older continuous kernel leads to RKHS functions that are $\alpha/2$-H\"older continuous".
\end{enumerate}
\end{remark}

\subsection{Converse results}
In \Cref{sec:lipschitz:kernelMetric} we saw that every RKHS function $f\in H_k$ is Lipschitz continuous w.r.t. $d_k$ with Lipschitz constant $\|f\|_k$.
Furthermore, in \Cref{sec:lipschitz:hoelderContMetricSpace} results were presented that ensure that RKHS functions are H\"older continuous w.r.t. a given metric on the input set, if the kernel fulfills a certain continuity condition.
But what about the converse? Assume we have a Hilbert function space $H$ such that all $f\in H$ are Lipschitz continuous (or H\"older continous) w.r.t. a given metric and Lipschitz (or H\"older) constant $\|f\|_H$. What can we say about $H$? And if $H$ is an RKHS, what can we say about the kernel?
To the best of our knowledge, these questions have not been addressed so far.

In this subsection, let $(\inputset{X},d_\inputset{X})$ be a metric space and $H \subseteq \K^\inputset{X}$ a Hilbert space of functions.
\begin{assumption} \label{assump:lipschitz:converse}
There exists $\alpha\in\Rp$ such that all $f\in H$ are $\alpha$-H\"older continuous with H\"older constant $\|f\|_H$.
\end{assumption}
\begin{proposition}
Suppose \Cref{assump:lipschitz:converse} holds, and that $H$ is an RKHS. Furthermore, let $k$ be the uniquely determined kernel with $H_k=H$.
\begin{enumerate}
    \item For all $x\in\inputset{X}$, $k(\cdot,x)\in H$ is $\alpha$-H\"older continuous with H\"older constant $\sqrt{k(x,x)}$.
    If $k$ is bounded, then $k(\cdot,x)$ is  $\alpha$-H\"older continuous with H\"older constant $\|k\|_\infty$, for all $x\in \inputset{X}$.
    \item For all $x_1,x_1',x_2,x_2'\in\inputset{X}$, {\tiny
    \begin{equation}
        |k(x_1,x_2)-k(x_1',x_2')| \leq \min\left\{
            \max\{\sqrt{k(x_2,x_2)}, \sqrt{k(x_1',x_1')}\}, 
            \max\{\sqrt{k(x_1,x_1)}, \sqrt{k(x_2',x_2')}\},
            \right\}(d_\inputset{X}(x_1,x_1')^\alpha+ d_\inputset{X}(x_2,x_2')^\alpha).
    \end{equation} }
    If $k$ is bounded, then 
    \begin{equation}
        |k(x_1,x_2)-k(x_1',x_2')| \leq \|k\|_\infty (d_\inputset{X}(x_1,x_1')^\alpha+ d_\inputset{X}(x_2,x_2')^\alpha)
    \end{equation}
    for all  $x_1,x_1',x_2,x_2'\in\inputset{X}$.
    \item For all $x,x'\in\inputset{X}$,
    \begin{equation}
        d_k(x,x') \leq \sqrt{\sqrt{k(x,x)}+\sqrt{k(x',x')}}d(x, x')^{\frac{\alpha}{2}}.
    \end{equation} 
    If $k$ is bounded, then 
    \begin{equation}
        d_k(x,x') \leq \sqrt{2\|k\|_\infty}d(x, x')^{\frac{\alpha}{2}}.
    \end{equation} 
    \item If $(\fs,\fm)$ is any feature space-feature map-pair, and $k$ is bounded, then $\fm$ is $\frac{\alpha}{2}$-H\"older continuous with H\"older constant $\sqrt{2\|k\|_\infty}$.
\end{enumerate}
\end{proposition}
\begin{proof}
The first claim follows immediately from \Cref{assump:lipschitz:converse} and the fact that $\|k(\cdot,x)\|_k=\sqrt{k(x,x)}$ for all $x\in\inputset{X}$, and the definition of $\|k\|_\infty$.

Let $x_1,x_1',x_2,x_2'\in\inputset{X}$ be arbitrary. Using \Cref{properties:lipschitz:separateHoelderAndHermitianImpliesHoelder} leads to
\begin{align*}
    |k(x_1,x_2)-k(x_1',x_2')| & \leq \sqrt{k(x_2,x_2)}d_\inputset{X}(x_1,x_1') + \sqrt{k(x_1',x_1')}d_\inputset{X}(x_2,x_2') \\
    & \leq \max\left\{\sqrt{k(x_2,x_2)}, \sqrt{k(x_1',x_1')} \right\},
\end{align*}
and repeating this computing with $x_1,x_2'$ instead of $x_2,x_1'$ establishes the second assertion.
Additionally,
\begin{align*}
    d_k(x,x') & = \sqrt{k(x,x)-k(x,x')-k(x',x)+k(x',x')} \\
    & \leq \sqrt{|k(x,x)-k(x',x)| + |k(x,x')-k(x',x')|} \\
    & \leq \sqrt{\sqrt{k(x,x)} + \sqrt{k(x',x')}d_\inputset{X}(x,x')^\alpha},
\end{align*}
showing the third claim.
This also establishes the last assertion, since for any feature space-feature map pair $(\fs,\fm)$ and all $x,x'\in\inputset{X}$ we have $\|\fm(x)-\fm(x')\|_\fs=d_k(x,x')$.
\end{proof}
\begin{corollary}
Assume that all $f\in H$ are Lipschitz continuous with Lipschitz constant $\|f\|_H$, that $H$ is an RKHS, and that the uniquely determined kernel $k$ with $H_k=H$ is bounded. Then $k$ is Lipschitz continuous with Lipschitz constant $\|k\|_\infty$.
\end{corollary}
The following result provides a simple condition for $H$ to be an RKHS, if $H$ fulfills \Cref{assump:lipschitz:converse}.
\begin{proposition} \label{properties:lipschitz:converseRKHS}
Suppose \Cref{assump:lipschitz:converse} holds, and that there exists $x_0\in\inputset{X}$ such that $f(x_0)=0$ for all $f\in H$.
In this case, $H$ is an RKHS.
Furthermore, $\sqrt{k(x,x)}\leq d_\inputset{X}(x,x_0)$ for all $x\in\inputset{X}$, where $k$ is the uniquely determined reproducing kernel of $H$.
\end{proposition}
\begin{proof}
Let $x\in\inputset{X}$ and consider the corresponding evaluation functional $\delta_x: H \rightarrow \K$, $\delta_x f = f(x)$.
We then have for all $f\in\inputset{X}$ that
\begin{equation*}
    |\delta_x f| = |f(x)|=|f(x)-f(x_0)| \leq \|f\|_H d_\inputset{X}(x,x_0),
\end{equation*}
which shows that $\delta_x$ is continuous, and $\|\delta_x\|\leq d_\inputset{X}(x,x_0)$.
Therefore, $H$ is an RKHS. Let $k$ be its uniquely determined reproducing kernel, then
\begin{equation*}
    \sqrt{k(x,x)}=\|k(\cdot,x)\|_H = \|\delta_x\| \leq d(x,x_0),
\end{equation*}
since $k(\cdot,x)$ is the uniquely determined Riesz representer of $\delta_x$ in $H$.
\end{proof}
Combining \Cref{properties:lipschitz:converseRKHS} with \Cref{properties:lipschitz:boundednessEquiv} leads to the following result.
\begin{corollary}
Assume that all $f\in H$ are bounded and Lipschitz continuous with Lispchitz constant $\|f\|_H$. Then $H$ is an RKHS with a bounded and Lipschitz continuous kernel $k$ having Lipschitz constant $\|k\|_\infty$.
\end{corollary}

In RKHSs, \Cref{assump:lipschitz:converse} can be relaxed.
\begin{lemma}
Let $k:\inputset{X}\times\inputset{X}\rightarrow\K$ be a kernel and $H_k$ its RKHS. Let $D\subseteq H_k$ be dense, and assume that there exists $\alpha\in\Rp$ such that all $f\in D$ are $\alpha$-H\"older continuous w.r.t. $d_\inputset{X}$ with H\"older bound $\|f\|_k$. Then all $f\in H_k$ are $\alpha$-H\"older continuous with H\"older bound $\|f\|_k$.
\end{lemma}
\begin{proof}
Let $f\in H_k$ and $x,x'\in \inputset{X}$ be arbitrary. Since $D$ is dense in $H_k$, there exists $(f_n)_{n\in\Np}\subseteq D$ such that $f_n \rightarrow f$ (in $H_k$). We then have
\begin{align*}
    |f(x)-f(x')| & = |\langle f, k(\cdot,x) - k(\cdot,x')\rangle_k| \\
    & = |\langle \lim_{n\rightarrow \infty} f_n, k(\cdot,x) - k(\cdot,x')\rangle_k| \\
    & =  \lim_{n\rightarrow \infty} |\langle  f_n, k(\cdot,x) - k(\cdot,x')\rangle_k| \\
    & = \lim_{n\rightarrow \infty} |f_n(x) - f_n(x')| \\
    & \leq \lim_{n\rightarrow \infty} \|f_n\|_k d(x,x')^\alpha \\
    & = \|f\|_k d(x,x')^\alpha.
\end{align*}
\end{proof}
Finally, under an additional assumption on $d_\inputset{X}$, \Cref{assump:lipschitz:converse} implies the existence of an RKHS on $H$.
The construction is classical, cf. \cite[Chapter~I]{atteia1992hilbertian}, but has not been used in this context before.

Suppose that  \Cref{assump:lipschitz:converse} holds and that $d_\inputset{X}$ is a \emph{Hilbertian metric}, i.e., there exists a $\K$-Hilbert space $\fs$ and a map $\fm:\inputset{X}\rightarrow\fs$, such that $d_\inputset{X}(x,x')=\|\fm(x)-\fm(x')\|_\fs$.

Define $\fs_0 = \{ \fm(x) \mid x \in \inputset{X}\}\subseteq\fs$, and for $f\in H$ set $\ell_f: \fs_0 \rightarrow \K$ by $\ell_f(\Phi(x))=f(x)$.
\begin{lemma}
    For all $f\in H$, $\ell_f$ as above is a well-defined, linear and continuous map.
\end{lemma}
\begin{proof}
Let $f\in H$ be arbitrary. In order to show that $\ell_f$ is well-defined, let $x,x'\inputset{X}$ such that $\fm(x)=\fm(x')$. We then have
\begin{equation*}
    |\ell_f(\Phi(x)) - \ell_f(\Phi(x'))| = |f(x)-f(x')| \leq \|f\|_H d_\inputset{X}(x,x')^\alpha 
    = \|f\|_H \|\Phi(x)-\Phi(x')\|_\fs^\alpha =0,
\end{equation*}
so $\ell_f(\Phi(x))=\ell_f(\Phi(x'))$, and $\ell_f$ is indeed well-defined.
Linearity and continuity are now clear.
\end{proof}
Given $f\in H$, we can now extend $\ell_f$ linearly to $\tilde{\ell_f}: \vspan \fs_0 \rightarrow \K$, and the resulting map is still well-defined, linear and continuous.
Define now $\fs_\inputset{X}=\overline{\vspan \fs_0}^{\|\cdot\|_\fs}$, then by construction $\fs_0$ is dense in $\fs_\inputset{X}$.
This means that for all $f\in H$, there exists a unique linear and continuous extension $\overline{\ell_f}: \fs_\inputset{X}\rightarrow\K$ of $\tilde{\ell_f}$. Note that this means that for all $f\in H$, $\overline{\ell_f}\in \fs_\inputset{X}'$ (the topological dual of $\fs_\inputset{X}$).
Since $\fs_\inputset{X}$ is itself a Hilbert space (because it is a closed subset of a Hilbert space), for each $f\in H$, there exists a unique Riesz representer $R(\ell_f)\in \fs_\inputset{X}$.
Define for all $f_1,f_2\in H$
\begin{equation}
    k(f_1,f_2)=\langle R(\ell_{f_2}), R(\ell_{f_1})\rangle_{\fs_\inputset{X}},
\end{equation}
then $k$ is a kernel on $H$ with feature space $\fs_\inputset{X}$ and feature map $H\ni f \mapsto R(\ell_f) \in \fs_\inputset{X}$.
The corresponding RKHS of $k$ is given by
\begin{equation}
    H_k = \{ f \mapsto \ell_f h \mid h \in \fs_\inputset{X} \},
\end{equation}
cf. \cite[Theorem~6.21]{SC08}.

\section{Lipschitz and H\"older continuity inducing kernels} \label{sec:lipschitz:inducing}
Essentially, the results in \Cref{sec:lipschitz:hoelderContMetricSpace} ensure that RKHS functions of $\alpha$-H\"older continuous kernels are $\alpha/2$-H\"older continuous. In particular, these results do not guarantee that RKHS functions of Lipschitz continuous kernels are themselves Lipschitz continuous.
However, for many applications the regularity properties (here Lipschitz and H\"older continuity) of RKHS functions matter most, and a kernel should be chosen that enforces the desired regularity properties for the induced RKHS functions.
This motivates the investigation of kernels that \emph{induce} prescribed H\"older continuity of its RKHS functions.
\subsection{Series expansions}
We start by characterizing \emph{all kernels} on a given metric space that have RKHS functions with prescribed H\"older continuity. To the best of our knowledge, this result is new.
\begin{theorem} \label{properties:lipschitz:characterizationHoelderInducingViaSeries}
Let $(\inputset{X},d_\inputset{X})$ be a metric space, $k$ a kernel on $\inputset{X}$, and $\alpha\in\Rp$. The following statements are equivalent.
\begin{enumerate}
    \item There exists $C\in\Rp$ such that all $f\in H_k$ are $\alpha$-H\"older continuous with H\"older constant $C\|f\|_k$.
    \item There exists a Parseval frame $(f_i)_{i\in I}$ in $H_k$, such that for all $i\in I$, $f_i$ is $\alpha$-H\"older continuous with H\"older constant $L_i\in\Rnn$, and $\sup_{i\in I} L_i < \infty$.
    \item There exists a family of functions $(f_i)_{i\in I}$, $f_i: \inputset{X}\rightarrow \K$, such that for all $i\in I$, $f_i$ is $\alpha$-H\"older continuous with H\"older constant $L_i\in\Rnn$, and $\sup_{i\in I} L_i < \infty$, and for all $x,x'\in \inputset{X}$
    \begin{equation}
        k(x,x') = \sum_{i\in I} f_i(x)\overline{f_i(x')},
    \end{equation}
    where the convergence is pointwise.
\end{enumerate}
\end{theorem}
\begin{proof}
\emph{2 $\Rightarrow$ 1}
Let $(f_i)_{i\in I}$ be a Parseval frame in $H_k$, such that for all $i\in I$, $f_i$ is $\alpha$-H\"older continuous with H\"older constant $L_i\in\Rnn$, and $\sup_{i\in I} L_i < \infty$.
Let $f\in H_k$ and $x,x'\in\inputset{X}$ be arbitrary, then we have
\begin{align*}
    |f(x)-f(x')| & = \left|\sum_{i\in I} \langle f, f_i\rangle_k f_i(x) - \sum_{i\in I} \langle f, f_i\rangle_k f_i(x') \right| \\
    & = \left|\sum_{i\in I} \langle f, f_i\rangle_k (f_i(x) -  f_i(x')) \right| \\
    & \leq \sum_{i\in I} |\langle f, f_i\rangle_k| |f_i(x) - f_i(x')| \\
    & \leq \sum_{i\in I} |\langle f, f_i\rangle_k| L_i d_\inputset{X}(x,x')^\alpha \\
    & \leq \left(\sum_{i\in I} |\langle f, f_i\rangle_k|\right) \left(\sup_{i\in I} L_i \right) d_\inputset{X}(x,x')^\alpha \\
    & \leq \sqrt{\sum_{i\in I} |\langle f, f_i\rangle_k|^2}\left(\sup_{i\in I} L_i \right) d_\inputset{X}(x,x')^\alpha \\
    & = \|f\|_k \left(\sup_{i\in I} L_i \right) d_\inputset{X}(x,x')^\alpha.
\end{align*}
In the first inequality we used that $(f_i)_{i\in I}$ is a Parseval frame, and that norm convergence (in $H_k$) implies pointwise convergence.
For the first inequality, we used the triangle inequality, and for the second inequality we used the assumption that $f_i$ is $\alpha$-H\"older continuous with H\"older constant $L_i$.
In the last inequality, we used 
\begin{align*}
    \sum_{i\in I} |\langle f, f_i\rangle_k| & = \|(\langle f, f_i\rangle_k)_{i\in I}\|_{\ell_1(I)} \\
    & \leq  \|(\langle f, f_i\rangle_k)_{i\in I}\|_{\ell_2(I)} \\
    & = \sqrt{\sum_{i\in I} |\langle f, f_i\rangle_k|^2}.
\end{align*}

\emph{2 $\Rightarrow$ 1}
Let $(e_i)_{i\in I}$ be an ONB of $H_k$, so $\|e_i\|_k=1$ for all $i\in I$. By assumption, all $e_i$ are $\alpha$-H\"older continuous with H\"older constant $1$, and since an ONB is a Parseval frame, the claim follows.

\emph{2 $\Rightarrow$ 3}
This implication follows immediately from \Cref{thm:papadakis}.

\emph{3 $\Rightarrow$ 2}
Let $(f_i)_{i\in I}$ be a family of function as given in the third item. By \Cref{thm:papadakis}, $f_i\in H_k$ for all $i\in I$, and $(f_i)_{i\in I}$ forms a Parseval frame, so this family of functions fulfills the conditions in the second item.
\end{proof}
Since orthonormal bases (ONBs) are Parseval frames, we get immediately the following result.
\begin{corollary}
Let $(\inputset{X},d_\inputset{X})$ be a metric space, $k$ a kernel on $\inputset{X}$, and $\alpha\in\Rp$. The following statements are equivalent.
\begin{enumerate}
    \item All $f\in H_k$ are $\alpha$-H\"older continuous with H\"older constant $\|f\|_k$.
    \item There exists an ONB $(e_i)_{i\in I}$ in $H_k$ such that for all $i\in I$, $e_i$ is $\alpha$-H\"older continuous with H\"older constant 1.
    \item For all ONB $(e_i)_{i\in I}$ in $H_k$, and all $i\in I$, $e_i$ is $\alpha$-H\"older continuous with H\"older constant 1.
    \item For all $x,x'\in \inputset{X}$,
    \begin{equation}
        k(x,x') = \sum_{i\in I} e_i(x)\overline{e_i(x')},
    \end{equation}
    where the convergence is pointwise, and $(e_i)_{i\in I}$ is an ONB $(e_i)_{i\in I}$ in $H_k$ such that for all $i\in I$, $e_i$ is $\alpha$-H\"older-continuous with H\"older constant 1.
\end{enumerate}
\end{corollary}
\subsection{Ranges of integral operators}
It is well-known that there is a close connection between the theory of RKHSs and integral operators.
For example, for RKHSs defined on measure spaces and under suitable technical assumptions, Mercer's theorem allows a spectral decomposition of the reproducing kernel, and an explicit description of the RKHS in terms of eigenfunctions of a related integral operator. For details, we refer to \cite[Section~4.5]{SC08}.
Moreover, integral operators defined using the reproducing kernel of an RKHS can have ranges contained in the RKHS under suitable assumptions, cf. \cite[Theorem~6.26]{SC08}. This motivates the study of H\"older continuity properties for functions in the image set of integral operators.
\paragraph{A general result} Before embarking on this task, we present a result for rather general integral maps. It is essentially a direct generalization of \cite[Theorem~5.1]{ferreira2013positive}.
\begin{proposition} \label{properties:lipschitz:hoelderContinuityIntOpRange}
Let $(\inputset{Y},\setsys{A},\mu)$ be a measure space, $(\inputset{X},d_\inputset{X})$ a metric space, $1 < p,q < \infty$ with $1/p+1/q=1$, and $k:\inputset{X}\times\inputset{Y}\rightarrow\K$ a function such that the following holds.
\begin{enumerate}
    \item For all $x\in\inputset{X}$, the function $k(x,\cdot)$ is measurable.
    \item For all $g \in L^q(\inputset{Y}, \setsys{A},\mu,\K)$ and all $x\in\inputset{X}$, $k(x,\cdot)\cdot g \in L^1(\inputset{Y}, \setsys{A},\mu,\K)$.
    \item There exists $\alpha\in\Rp$, $L_\alpha \in \Ellp^p(\inputset{Y}, \setsys{A},\mu,\Rnn)$, such that for $\mu$-almost all $y\in\inputset{Y}$, the function $k(\cdot,y)$ is $\alpha$-H\"older continuous with H\"older constant $L_\alpha(y)$.
\end{enumerate}
In this case,
\begin{equation}
    S_k:  L^q(\inputset{Y}, \setsys{A},\mu,\K) \rightarrow \K^\inputset{X}, \quad
    (S_k g)(x) = \int_\inputset{Y} k(x,y)g(y)\mathrm{d}\mu(y)
\end{equation}
is a well-defined linear mapping, and for all $g\in L^q(\inputset{Y}, \setsys{A},\mu,\K)$, the function $f=S_k g$ is $\alpha$-H\"older continuous with H\"older constant $\|L_\alpha\|_{\Ellp^p}\|g\|_{L^q}$.
\end{proposition}
\begin{proof}
Since for all $g \in L^q(\inputset{Y}, \setsys{A},\mu,\K)$ and all $x\in\inputset{X}$ the function $k(x,\cdot)g\in  L^1(\inputset{Y}, \setsys{A},\mu,\K)$, the mapping $S_k$ is well-defined. The linearity is now clear.

Let $g \in L^q(\inputset{Y}, \setsys{A},\mu,\K)$, define $f=S_k g$, and let $x,x'\in\inputset{X}$ be arbitrary, then
\begin{align*}
    |f(x)-f(x')| & = \left| \int_\inputset{Y} (k(x,y) - k(x',y))g(y)\mathrm{d}\mu(y)\right| \\
    & \leq  \int_\inputset{Y} |k(x,y) - k(x',y)| |g(y)| \mathrm{d}\mu(y) \\
    & \leq  \int_\inputset{Y} L_\alpha(y) |g(y)| \mathrm{d}\mu(y)d_\inputset{X}(x,x') \\
    & \leq \|L_\alpha\|_{\Ellp^p} \|g\|_{L^q} d_\inputset{X}(x,x'),
\end{align*}
so $f$ is indeed $\alpha$-H\"older continuous with H\"older constant $\|L_\alpha\|_{\Ellp^p}\|g\|_{L^q}$.
\end{proof}

\paragraph{Example}
To illustrate \Cref{properties:lipschitz:hoelderContinuityIntOpRange}, we consider the rather general class of integral operators described in \cite[Abschnitt~6.3]{weidmann2000lineare}.
Let $(\inputset{X},\setsys{A}_\inputset{X},\mu)$ and $(\inputset{Y},\setsys{A}_\inputset{Y},\nu)$ be measure spaces, $1 < p,q < \infty$ with $1/p+1/q=1$, and $k: \inputset{X}\times\inputset{Y}\rightarrow\K$ be measurable. Assume that for all $g\in L^q(\inputset{Y},\setsys{A}_\inputset{Y},\nu)$ and $\mu$-almost all $x\in\inputset{X}$, $k(x,\cdot)g\in L^1(\inputset{Y},\setsys{A}_\inputset{Y},\nu)$, and that by defining ($\mu$-almost all) $x\in\inputset{X}$
\begin{equation}
    (T_k g)(x) = \int_\inputset{Y} k(x,y)g(y)\mathrm{d}\nu(y)
\end{equation}
we get $T_kg \in L^p(\inputset{X},\setsys{A}_\inputset{X},\mu)$.
Under these conditions, $T_k: L^q(\inputset{Y},\setsys{A}_\inputset{Y},\nu) \rightarrow  L^p(\inputset{X},\setsys{A}_\inputset{X},\mu)$ is a well-defined, linear and bounded operator.

Assume furthermore that $(\inputset{X},d_\inputset{X})$ is a metric space, and that there exists $\alpha\in\Rp$ and  $L_\alpha \in \Ellp^p(\inputset{Y}, \setsys{A},\mu,\Rnn)$, such that for $\mu$-almost all $y\in\inputset{Y}$, the function $k(\cdot,y)$ is $\alpha$-H\"older continuous with H\"older constant $L_\alpha(y)$. 
Let $g\in  L^q(\inputset{Y},\setsys{A}_\inputset{Y},\nu)$, then there exists a $\mu$-nullset $\inputset{N}_g$ such that (setting for brevity $\inputset{X}_g=\inputset{X}\setminus\inputset{N}_g$) $f: \inputset{X}_g \rightarrow \K$, $f(x)=(T_k g)(x)$ is well-defined.
\Cref{properties:lipschitz:hoelderContinuityIntOpRange} now ensures that $f$ is $\alpha$-H\"older continuous with H\"older constant $\|L_\alpha\|_{\Ellp^p}\|g\|_{L^q}$, though $f$ is only defined on the restricted metric space $(\inputset{X}_g, d_\inputset{X}\lvert_{\inputset{X}_g \times \inputset{X}_g})$.

In particular, each element\footnote{Recall that this is an equivalence class of functions on $\inputset{X}$.} of the image set of $T_k$ contains a $\mu$-almost everywhere defined function that is $\alpha$-H\"older continuous.

We can strengthen this result. Let $\setsys{A}_\inputset{X}$ be the Borel $\sigma$-algebra on $\inputset{X}$, and assume that $\mu(U)>0$ for all open nonempty $U\subseteq \inputset{X}$. In this case, $\inputset{X}_g$ is dense in $\inputset{X}$, since otherwise $\inputset{N}_g$ contains a nonempty open set $U$, and hence $\mu(\inputset{N}_g) \geq \mu(U) >0$, a contradiction to the fact that $\inputset{N}_g$ is a $\mu$-nullset.
Since $f$ is defined on a dense subsetset of $\inputset{X}$, and it is continuous (since it is $\alpha$-H\"older continuous on $\inputset{X}_g$), there exists a unique extension $\bar f: \inputset{X}\rightarrow\K$ that is also $\alpha$-H\"older continuous.
Defining $\bar T_k g := \bar f$, we thus arrived at a linear operator from $L^q(\inputset{Y},\setsys{A}_\inputset{Y},\nu)$ into $\Ellp(\inputset{X},\setsys{A}_\inputset{X},\mu)$ with its range space consisting of $\alpha$-H\"older continuous functions.

\paragraph{Integral operators into RKHSs} Let us return to the setting of RKHSs. If an RKHS is defined on a measure space, and the kernel fulfills an integrability condition, then the RKHS consists of integrable functions, and the kernel allows the definition of a related integral operator with range contained in the RKHS.
The next result provides a sufficient condition for H\"older continuity of RKHS functions in the range of this integral operator.
\begin{proposition}
Let $(\inputset{X},d_\inputset{X})$ be a metric space, $(\inputset{X},\setsys{A},\mu)$ a $\sigma$-finite measure space,\footnote{$\setsys{A}$ can, but does not have to be the Borel $\sigma$-algebra on the metric space $\inputset{X}$.} $1<p,q<\infty$ with $1/p+1/q=1$, and $k:\inputset{X}\times\inputset{X}\rightarrow\K$ a measurable kernel such that $H_k$ is separable and
\begin{equation}
    \|k\|_{L^p} = \left(\int (k(x,x))^{\frac{p}{2}}\mathrm{d}\mu(x)\right)^{\frac{1}{p}} < \infty.
\end{equation}
Assume that there exist $\alpha\in\Rp$, $L_\alpha \in \Ellp^p(\inputset{X},\setsys{A},\mu,\Rnn)$ such that for $\mu$-almost all $x\in\inputset{X}$ the function $k(\cdot,x)$ is $\alpha$-H\"older continuous with H\"older constant $L_\alpha(x)$.

Under these conditions, 
\begin{equation}
    S_k: L^q(\inputset{X},\setsys{A},\mu,\K) \rightarrow H_k, \quad
    (S_k g)(x) = \int_\inputset{X} k(x,x')g(x')\mathrm{d}\mu(x')
\end{equation}
is a well-defined, bounded linear operator, and for all $g\in L^q(\inputset{X},\setsys{A},\mu,\K)$, the function $f=S_k g \in H_k$ is $\alpha$-H\"older continuous with H\"older constant $\|L_\alpha\|_{\Ellp^p}\|g\|_{L^q}$.

Finally, all functions in $H_k$ are $p$-integrable,\footnote{This means that for all $f\in H_k$, $\int_\inputset{X} |f(x)|^p\mathrm{d}\mu(x) <\infty$.} and if the inclusion $\identity: H_k \rightarrow L^p(\inputset{X},\setsys{A},\mu,\K)$ is injective, then the image of $S_k$ is dense in $H_k$.
\end{proposition}
\begin{proof}
That $S_k$ is well-defined, linear and bounded, follows from \cite[Theorem~6.26]{SC08}.
The statement on the H\"older continuity of the functions in the images of $S_k$ is a direct consequence of \Cref{properties:lipschitz:hoelderContinuityIntOpRange}.
The last claim follows again from \cite[Theorem~6.26]{SC08}.
\end{proof}
\subsection{Feature mixture kernels}
\Cref{properties:lipschitz:characterizationHoelderInducingViaSeries} characterizes H\"older continuity inducing kernels via series expansion. However, these might be difficult to work with, so an alternative description of such kernels can be useful. The next result presents a very general construction which is based on a mixture of feature maps. It vastly generalizes a method apparently introduced in \cite{wu2018d2ke}.
\begin{theorem} \label{properties:lipschitz:featureMixtureHoelderInducing}
    Let $(\Omega,\setsys{A})$ be a measurable space, $\mu$ a finite nonnegative measure on $(\Omega,\setsys{A})$, $(\inputset{X},d_\inputset{X})$ a metric space, and $\fs$ a $\K$-Hilbert space.
    Furthermore, let $\fm(x,\cdot)\in\Ellp^2(\Omega,\setsys{A},\mu,\fs)$ for all $x\in\inputset{X}$.
    Finally, assume that there exist $\alpha,\:L_\fm\in\Rp$ such that for $\mu$-almost all $\omega\in\Omega$, $\fm(\cdot,\omega)$ is $\alpha$-H\"older continuous with H\"older constant $L_\fm$.
    Then
    \begin{equation}
        k(x,x') = \int_\Omega \langle \fm(x',\omega), \fm(x,\omega) \rangle_\fs \mathrm{d}\mu(\omega)
    \end{equation}
    is a well-defined kernel on $\inputset{X}$, and all $f\in H_k$ are $\alpha$-H\"older continuous with H\"older constant $L_\fm \sqrt{\mu(\Omega)}\|f\|_k$.
\end{theorem}
\begin{proof}
First, we show that $k$ is well-defined. Let $x,x'\in\inputset{X}$, then $\|\Phi(x,\cdot)\|_\fs, \|\Phi(x',\cdot)\|_\fs$ are square-integrable, so we get 
\begin{align*}
    \int_\Omega | \langle \fm(x',\omega), \fm(x,\omega) \rangle_\fs| \mathrm{d}\mu(\omega)
    & \leq \int_\Omega \|\fm(x,\omega)\|_\fs \|\fm(x',\omega)\|_\fs \mathrm{d}\mu(\omega) \\
    & \leq \left( \int_\Omega \|\Phi(x,\omega)\|_\fs^2 \mathrm{d}\mu(\omega)\right)^\frac12 \left( \int_\Omega \|\Phi(x',\omega)\|_\fs^2 \mathrm{d}\mu(\omega)\right)^\frac12 < \infty,
\end{align*}
where we  used Cauchy-Schwarz first in $\fs$, then in $\Ellp^2$.

Next, we show that $k$ is kernel by verifying that it is positive semidefinite. Let $x_1,\ldots,x_N\in\inputset{X}$ and $c_1,\ldots,c_N\in\C$ be arbitrary, then
\begin{align*}
    \sum_{i,j=1}^N c_i \overline{c_j} k(x_j,x_i) & = \int_\Omega  \sum_{i,j=1}^N c_i \overline{c_j} \langle \fm(x_j,\omega), \fm(x_i,\omega) \rangle_\fs \mathrm{d}\mu(\omega) \\
    & = \int_\Omega \left\langle \sum_{i=1}^N c_i \fm(x_i,\omega), \sum_{j=1}^N c_j \fm(x_j,)\right\rangle_\fs \mathrm{d}\mu(\omega) \\
    & = \int_\Omega \left\|  \sum_{i=1}^N c_i \fm(x_i,\omega) \right\|_\fs^2\mathrm{d}\mu(\omega) \\
    & \geq 0,
\end{align*}
so $k$ is indeed positive semidefinite.

Finally, let $f\in H_k$ and $x,x'\in\inputset{X}$ be arbitrary, then $|f(x)-f(x')| \leq \|f\|_k d_k(x,x')$. Observe now that
\begin{align*}
    d_k(x,x')^2 & = k(x,x) + k(x,x') + k(x',x) + k(x',x') \\
    & = \int_\Omega \langle \fm(x,\omega), \fm(x,\omega)\rangle_\fs + \langle \fm(x,\omega), \fm(x',\omega)\rangle_\fs \\
    & \hspace{0.5cm} + \langle \fm(x',\omega), \fm(x,\omega)\rangle_\fs + \langle \fm(x',\omega), \fm(x',\omega)\rangle_\fs \mathrm{d}\mu(\omega) \\
    & = \int_\Omega \langle \fm(x,\omega)-\fm(x',\omega),  \fm(x,\omega)-\fm(x',\omega)\rangle_\fs \mathrm{d}\mu(\omega) \\
    & =  \int_\Omega \|\fm(x,\omega)-\fm(x',\omega)\|_\fs \mathrm{d}\mu(\omega) \\
    & \leq \int_\Omega L_\fm^2 d_\inputset{X}(x,x')^{2\alpha}\mathrm{d}\mu(\omega)\\
    & = L_\fm^2 \mu(\Omega) d_\inputset{X}(x,x')^{2\alpha},
\end{align*}
so we get
\begin{align*}
    |f(x)-f(x')| \leq \|f\|_k d_k(x,x') 
    \leq L_\fm \sqrt{\mu(\Omega)}\|f\|_k d_\inputset{X}(x,x')^\alpha.
\end{align*}
\end{proof}
If the nonnegative measure in the preceding result is a probability measure, we get the following result as a special case.
\begin{corollary}
Let $(\inputset{X},d_\inputset{X})$ be a metric space, $\fs$ a $\K$-Hilbert space, and $(\fm(x))_{x\in\inputset{X}}$ a family of square-integrable $\fm$-valued random variables.
Assume that there exist $\alpha,L_\fm\in\Rp$ such that $\Phi$ is almost surely $\alpha$-H\"older continuous with H\"older constant $L_\fm$. Then
\begin{equation}
    k(x,x') = \E[\langle \fm(x'), \fm(x)\rangle_\fs]
\end{equation}
is a well-defined kernel on $\inputset{X}$, and all $f\in H_k$ are $\alpha$-H\"older continuous with H\"older constant $L_\fm\|f\|_k$.
\end{corollary}
The importance of this result is the fact that the kernel $k$ described there is a \emph{random feature kernel} in the sense of \cite{rahimi2007random}. In particular, in practice $k(x,x')$ can be approximated by sampling from the random variables $\Phi(x),\Phi(x')$.

Finally, we can formulate another special case, which recovers the approach from \cite{wu2018d2ke}.
\begin{proposition}
Let  $(\inputset{X},d_\inputset{X})$ be a metric space, $P$ a Borel probability measure on $\inputset{X}$, $\varphi: \Rnn\rightarrow\K$ an $\alpha$-H\"older-continuous function with H\"older-constant $L_\varphi$, and define $\phi: \inputset{X}\times\inputset{X}\rightarrow\K$ by $\phi(x,z)=\varphi(d_\inputset{X}(x,z))$.
If $\phi(x,\cdot) \in \Ellp^2(\inputset{X},P)$ for all $x\in\inputset{X}$, then
\begin{equation}
    k(x,x') = \int_\inputset{X} \phi(x',z) \overline{\phi(x,z)}\mathrm{d}P(z)
\end{equation}
is a well-defined kernel on $\inputset{X}$, and all $f\in H_k$ are $\alpha$-H\"older continuous with H\"older constant $L_\varphi \|f\|_k$.
\end{proposition}
\begin{proof}
We show that for all $z\in\inputset{X}$, the function $\phi(\cdot,z)$ is $\alpha$-H\"older continuous with H\"older constant $L_\varphi$.
For this, let $x,x'\in\inputset{X}$ be arbitrary, then
\begin{align*}
    |\phi(x,z)-\phi(x',z)| & = |\varphi(d_\inputset{X}(x,z)) - \varphi(d_\inputset{X}(x',z))| \\
    & \leq L_\varphi |d_\inputset{X}(x,z)^\alpha - d_\inputset{X}(x',z)^\alpha| \\
    & \leq L_\varphi d_\inputset{X}(x,x')^\alpha,
\end{align*}
where we used the inverse triangle inequality for the metric $(x,x')\mapsto d_\inputset{X}(x,x')^\alpha$ in the last step.

The result follows now from \Cref{properties:lipschitz:featureMixtureHoelderInducing} by choosing $\Omega=\inputset{X}$, $\mu=P$, $\fs=\K$, and $\fm=\phi$, and the fact that $P(\inputset{X})=1$.
\end{proof}